\documentclass[12pt]{amsart}

\usepackage{amsmath,amsthm,amsfonts,amssymb,mathrsfs}
\usepackage[margin=1in]{geometry}
\usepackage[bookmarks]{hyperref}
\usepackage{verbatim}
\usepackage{todonotes}
\usepackage{tikz}
\usetikzlibrary{matrix,arrows}
\usepackage{fancyvrb }
\usepackage{color}
\usepackage[all,cmtip]{xy}

\newenvironment{enumalph}{\begin{enumerate}  }{\end{enumerate}}

\newtheorem{theorem}{Theorem}[section]

\newtheorem{proposition}[theorem]{Proposition}

\newtheorem{corollary}[theorem]{Corollary}

\newtheorem{lemma}[theorem]{Lemma}
\input xy
\xyoption{all}

\theoremstyle{definition}

\newtheorem{remark}[theorem]{Remark}

\DeclareMathOperator{\rank}{rank}

\DeclareMathOperator{\Ext}{Ext}

\DeclareMathOperator{\opH}{H}
\DeclareMathOperator{\Hom}{Hom}

\DeclareMathOperator{\ind}{ind}
\DeclareMathOperator{\Lie}{Lie}

\newcommand{\calO}{\mathcal{O}}

\newcommand{\spec}{\rm{Spec}}

\newcommand{\frakb}{\mathfrak{b}}
\newcommand{\g}{\mathfrak{g}}

\newcommand{\frakgl}{\mathfrak{gl}}
\newcommand{\fraku}{\mathfrak{u}}

\newcommand{\fraksl}{\mathfrak{sl}}

\newcommand{\N}{\mathcal{N}}

\newcommand{\calm}{\mathcal{M}}

\begin{document}

\title[Rational singularities of $G$-saturation]{Rational singularities of $G$-saturation}

\author{Nham V. Ngo}
\address{Department of Mathematics \\ University of Arizona \\ Tucson \\ AZ 85721, USA}
\email{nhamngo@math.arizona.edu}

\date{\today}

\maketitle

\begin{abstract}
Let $G$ be a semisimple algebraic group  defined over an algebraically closed field of characteristic 0 and $P$ be a parabolic subgroup of $G$. Let $M$ be a $P$-module and $V$ be a $P$-stable closed subvariety of $M$. We show in this paper that if the varieties $V$ and $G\cdot M$ have rational singularities, and the induction functor $R^i\ind_P^G(-)$ satisfies certain vanishing condition then the variety $G\cdot V$ has rational singularities. This generalizes the main result of Kempf in \cite{Kem:76}. As an application, we prove the property of having rational singularities for nilpotent commuting varieties over $3\times 3$ matrices.
\end{abstract}

\section{Introduction}
The study of rational singularities for varieties of dimension higher than two dates back to the 70's when Kempf investigated the geometry of a Riemann's theorem \cite{Kem:73}. One of interesting questions is when the $G$-saturation $G\cdot V$ preserves the property of having rational singularities of $V$. In particular, let $M$ be a $G$-module and $V$ be a closed subvariety of $M$ stabilized by a parabolic subgroup $P$ of $G$. Then Kempf proved in \cite{Kem:76}\cite{Kem:86} that if $P$ acts completely reducible to $V$ and $V$ has rational singularities then $G\cdot V$ has rational singularities. It is known that completely reducible actions rarely occur, so the usage of Kempf's result is rather restricted. We prove in the present paper that this condition can be relaxed, see Theorem \ref{maintheorem}. As an application, we show that nilpotent commuting varieties over $3\times 3$ matrices  have rational singularities. Our interest in nilpotent commuting varieties was motivated by their connection to the cohomology of Frobenius kernels of $G$ in positive characteristic, see \cite{N:2015}.


The paper is organized as follows. Section \ref{notation} provides necessary notation and background. The main result is shown in Section \ref{existence of equivariant rational resolution}. From Section \ref{vanishing} to the end of the paper, we assume $G=SL_3$. Before showing our applications, we prove in Section \ref{vanishing} a vanishing result of higher induction $R^i\ind_B^G(-)$ for certain modules. Computations in this section extend a recent result on the null-cone of Vilonen and Xue \cite{VX}. Next, in Section \ref{normality}, these vanishing results are applied to prove the rational singularities of the nilpotent commuting varieties $C_r(\N)$ and some related varieties.

\section{Notation}\label{notation}

\subsection{Algebraic groups and Lie algebras}\label{combinatorics} Let $k$ be an algebraically closed field of characteristic $0$. Let $G$ be a semisimple algebraic group defined over $k$, unless otherwise stated. Fix a maximal torus $T \subset G$, and let $\Phi$ be the root system of $T$ in $G$. Let $\Phi^+$ be the corresponding set of positive roots. Let $B$ be the Borel subgroup of $G$ containing $T$ and corresponding to the set of negative roots $\Phi^-$, and let $U$ be the unipotent radical of $B$. Set $\g = \Lie(G)$, the Lie algebra of $G$, $\frakb = \Lie(B)$, $\fraku = \Lie(U)$. 

Given a vector space $V$, we denote by $S^n(V)$ and $\Lambda^n(V)$ the symmetric and exterior space of degree $n$ over $V$. Then the direct sums 
\[ S(V)=\bigoplus_{n=0}^{\infty}S^n(V)\quad,\quad \Lambda(V)=\bigoplus_{n=0}^{\infty}\Lambda^n(V) \]
denote the symmetric algebra and exterior algebra of $V$. Define $V^*=\Hom_k(V,k)$ the dual space of $V$. Throughout this paper, tensor products will be taken over $k$. Assume for the rest of the paper that every $G$-module is a rational module over $G$.

\subsection{Induction functor} Let $M$ be a $P$-module where $P$ is a parabolic subgroup of $G$. Then the induced $G$-module can be defined as 
\[ \ind_P^GM=(k[G]\otimes M)^P. \]
The higher derived functor of $\ind_P^G(-)$ is denoted by $R^i\ind_P^G(-)$. Note that the definition of induction is not restricted to parabolic subgroups, we refer the reader to \cite[Chapter I.3]{Jan:2003} for further details. 

\subsection{Adjoint action} The group $G$ acts on the Lie algebra $\g$ via the adjoint action denoted by ``$\cdot$", and called  $G$-action. Note that the nilpotent cone $\N$ of $\g$ is stable under this $G$-action and $\frakb$, $\fraku$ are stable under the $B$-action, the restriction of $G$-action to $B$. For every positive integer $r$, the $G$-action on the direct product $\g^r$ is defined diagonally, i.e., \[ g\cdot(x_1,\ldots,x_r)=(g\cdot x_1,\ldots, g\cdot x_r)\] for all $g\in G$ and $x_i\in\g$. It also restricts to the $B$-action on $\frakb^r$ and $\fraku^r$. 

Let $X$ be a variety and $H$ be a connected algebraic group. Then $X$ is called an $H$-variety if it is stable under the $H$-action. This action induces an action on the coordinate algebra $k[X]$ so we call it an $H$-algebra. A morphism between two $H$-varieties $f:X\to Y$ is called {$H$-equivariant} if it commutes with the $H$-actions on both varieties. For instance, the moment map $m:G\times^B\fraku^r\to G\cdot\fraku^r$ defined by $(g,x_1,\ldots,x_r)\mapsto(g\cdot x_1,\ldots,g\cdot x_r)$ for all $g\in G, x_i\in\fraku$, is $G$-equivariant.     

\subsection{Basic algebraic geometry conventions}\label{algebraic geometry conventions} Let $X$ be a variety. We write $k[X]$ for the ring of global sections $\calO_X(X)$ on $X$. In case $X$ is affine, it coincides with the coordinate algebra of $X$.

 For each variety $X$, a morphism $\pi:\tilde{X}\to X$ is called a {\it resolution of singularites} if the variety $\tilde{X}$ is non-singular and $\pi$ is proper and birational. If in addition $X$ is normal and the higher direct image of $\pi$ vanishes, i.e., $R^i\pi_*\calO_{\tilde{X}}=0$ for all $i>0$, then we call $\pi$ a {\it rational resolution} and say that $X$ has rational singularities. Note that this vanishing condition is equivalent to $\opH^i(\tilde{X},\calO_{\tilde{X}})=0$ for all $i>0$ when $X$ is affine, \cite[Proposition III.8.5]{Ha:1977}. This notion can also be applied to a commutative ring $R$ if we replace $X$ by $\spec(R)$. Suppose further that $\pi$ is $H$-equivariant, then the resolution above is called $H$-equivariant resolution of singularities (resp. $H$-equivariant rational resolution). The below proposition about the existence of equivariant rational resolution of an $H$-variety should be well-known, however, we have not seen it in literature. 

\begin{proposition}\label{equiv resolution}
Let $H$ be a connected algebraic group and $X$ be an $H$-variety. If $X$ has rational singularities, then there exists a $H$-equivariant rational resolution of $X$.
\end{proposition}

\begin{proof}
First note that $X$ has an $H$-equivariant resolution of singularities, namely $\pi:\tilde{X}\to X$, see for example \cite[Proposition 3.9.1]{Ko}. On the other hand, the rational singularities of $X$ and Remark iv (or Lemma 1) in \cite{V:1977} imply that $\pi$ must carry the property of being rational singularities. 
\end{proof}

Next let $P$ be a parabolic subgroup of $G$. The associated bundle of a $P$-variety $X$ over $G/P$ is denoted by $G\times^PX$. It is known that the ring of global sections on $G\times^PX$ coincides with the ring of $P$-invariant global sections on $G\times X$. In particular, we have 
\[ k[G\times^PX]\cong k[G\times X]^P\cong (k[G]\otimes k[X])^P=\ind_P^Gk[X]. \]
Furthermore, we have for all $i\ge 0$ \[ \opH^i(G\times^P X,\calO_{G\times^P X})\cong R^i\ind_P^G(k[X]), \]
where the left-hand side is the sheaf cohomology of the scheme $G\times^PX$.

\subsection{Determinantal varieties}
Consider an $m\times n$ matrix
\[ \calm=\left( \begin{array}{ccc}
x_{11} & \cdots & x_{1n} \\
\vdots & \ddots & \vdots \\
x_{m1} & \cdots & x_{mn} \end{array} \right) \]
whose entries are independent indeterminates over the field $k$. Let 
\[k[\calm]:=k[x_{ij}:1\le i\le m, 1\le j\le n],\]
and let $I_t(\calm)$ be the ideal in $k[\calm]$ generated by all $t\times t$ minors of $\calm$. For each $t\ge 1$, the ring 
\[ R_t(\calm)=\frac{k[\calm]}{I_t(\calm)} \]
is called a {\it determinantal ring}. We denote by $D_t(\calm)$ the determinantal variety defined by $I_t(\calm)$. These rings (or varieties) are well-known in commutative algebra. To be convenient, we state some of their nice properties.
 
\begin{proposition}\cite[2.13, 11.23]{BV:1988}\label{commutative algebra}
For every $1\le t\le\min(m,n)$, the ring $R_t(\calm)$ (or the variety $D_t(\calm)$) is a reduced, Cohen-Macaulay, normal domain of dimension $(t-1)(m+n-t+1)$. Furthermore, it has rational singularities.
\end{proposition}

\section{Equivariant rational resolution}\label{existence of equivariant rational resolution}
We prove in this section the main result of the paper. Recall that $G$ is a connected semisimple algebraic group. The argument is a combination of techniques in \cite[\S 5]{KLT:1999} and ingredients in \cite{Br:2003}.

\begin{theorem}\label{maintheorem}
Let $V$ be a $P$-subvariety of a $P$-module $M$ contained in a $G$-module. Let $I(V)$ be the defining ideal of $V$ in $k[M]=S(M^*)$. Assume that 
\[ m:G\times^PM\to G\cdot M \]
is a rational resolution of $G\cdot M$. If $V$ is normal and for all $i\ge 1$
\[ R^i\ind_P^GI(V)=0,\]
then $G\cdot V$ is normal. 

Furthermore, suppose that $V$ has rational singularities. If the map $m':G\times^PV\to G\cdot V$, the restriction of $m$,  is a proper birational map, then the variety $G\cdot V$ has rational singularities.
\end{theorem}

The argument is split up to several steps as follows. 

\begin{lemma}\label{lemma 1}
Let $q$ be the embedding of $V$ into $M$ which induces a surjective homomorphism of $P$-algebras $q^*:S(M^*)\to k[V]$. Then the map
\[ \phi:=\ind_P^G(q^*):\ind_P^GS(M^*)\to \ind_P^Gk[V] \]
is a surjective $G$-equivariant homomorphism of algebras.
\end{lemma}

\begin{proof}
Note first that \[ k[V]\cong\frac{S(M^*)}{I(V)}\]
are isomorphisms of $P$-algebras. We then have the following short exact sequence of $P$-modules
\begin{align}\label{short exact sequence}
 0\to I(V)\to  S(M^*)\stackrel{q^*}{\longrightarrow} k[V]\to 0.
\end{align}
Since we have for all $i\ge 1$ 
\begin{align*}
R^i\ind_P^G\left(I(V)\right)=0,
\end{align*}
the long exact sequence when applying the induction functor to the short exact sequence \eqref{short exact sequence} deduces to the short exact sequence of $G$-modules
\begin{align*}
 0\to\ind_P^GI(V)\to\ind_P^GS(M^*)\stackrel{\phi}{\longrightarrow} \ind_P^G(k[V])\to 0,
\end{align*}
which implies the surjectivity of $\phi$.
\end{proof}

We now prove the first statement in the theorem.

\begin{lemma}\label{normal}
The variety $G\cdot V$ is normal.
\end{lemma}

\begin{proof}
As $V$ is normal, the ring \[\ind_P^Gk[V]\cong k[G\times^PV]\cong k[G\times V]^P\] is also normal. Hence, it suffices to show that the map
\begin{align*}
m'^*: k[G\cdot V]\to k[G\times^P V]
\end{align*}
is an isomorphism. Clearly, it is injective. To show the surjectivity of $m'^*$, we consider the commutative diagram of $G$-equivariant morphisms
\[
\xymatrix{
G\times^PV \ar@{->}[r]^-{m'} \ar@{->}[d]_{\overline{q}} & G\cdot V \ar@{->}[d]^{e} \\
G\times^PM \ar@{->}[r]^-{m} & G\cdot M
}
\]
where $\overline{q}$ is induced from the embedding $V\hookrightarrow M$, and $e$ is the embedding from $G\cdot V$ into $G\cdot M$. It follows the commutative diagram of $G$-algebras
\[
\xymatrix{
\ind_P^Gk[V]\quad\quad \cong & k[G\times^PV] & \ar@{->}[l]_--{m'^*} k[G\cdot V]  \\
\ar@{->>}[u]^{\phi}\ind_P^GS(M^*)\quad\quad \cong & \ar@{->}[u]^{\overline{q}^*}k[G\times^PM] & \ar@{->}[l]_-{m^*} k[G\cdot M].\ar@{->}[u]^{e^*}
}
\]
We have, by Lemma \ref{lemma 1}, $\phi$ is onto. Also, $m^*$ is an isomorphism since $m:G\times^PM\to G\cdot M$ is a rational resolution. Therefore, $m'^*$ is surjective, which proves our lemma.
\end{proof}

Before completing the proof, we need to set up a few things. By Proposition \ref{equiv resolution}, $V$ has a $P$-equivariant rational resolution, namely $\pi:\tilde{V}\to V$. This morphism can be extended to the birational map $\tilde{\pi}:G\times^P\tilde{V}\to G\times^P{V}$ by setting $(g,v)\mapsto(g,\pi(v))$ for all $g\in G, v\in V$. Then composing with the map $m':G\times^P V\to G\cdot V$, we have a $G$-equivariant resolution of singularities
\begin{align}\label{GV_resolution}
 m'\circ\tilde{\pi}:G\times^P\tilde{V}\to G\cdot V. 
\end{align}

{\bf Proof of Theorem \ref{maintheorem}.}

We are showing that \ref{GV_resolution} is a rational resolution of $G\cdot V$. By Lemma \ref{normal}, we only need to show that 
\[ R^i(m'\circ\tilde{\pi})_*\calO_{G\times^P\tilde{V}}=0\]
for all $i\ge 1$. Using \cite[Proposition III.8.5]{Ha:1977} and \cite[Lemma 2]{Br:2003}, we have, for all $i\ge 1$, the following
\begin{align*} 
R^i(m'\circ\tilde{\pi})_*\calO_{G\times^P\tilde{V}} &\cong\opH^i(G\times^P\tilde{V},\calO_{G\times^P\tilde{V}})^{\sim}\\
&\cong \left(\Ext^i_{G\times^P\tilde{V}}(\calO_{G\times^P\tilde{V}},\calO_{G\times^P\tilde{V}})\right)^\sim\\
&\cong\left(\Ext^i_{G\times^P\tilde{V}}(G\times^P\calO_{\tilde{V}}, G\times^P\calO_{\tilde{V}})\right)^\sim\\
&\cong \left(G\times^P\Ext^i_{\tilde{V}}(\calO_{\tilde{V}}, \calO_{\tilde{V}})\right)^\sim\\
&\cong \left(G\times^P\opH^i(\tilde{V},\calO_{\tilde{V}})\right)^\sim\\
&=0
\end{align*}
since $\opH^i(\tilde{V},\calO_{\tilde{V}})^{\sim}\cong R^i\pi_*\calO_{\tilde{V}}=0$ for all $i\ge 1$ (by the rational resolution $\pi:\tilde{V}\to V$). Hence, the theorem is proved.

Now we are able to retrieve the main result of Kempf in \cite{Kem:76}.

\begin{corollary}
Let $N$ be a $P$-module contained in a $G$-module $M$. Suppose $P$ acts on $N$ completely reducible. Then $G\cdot N$ is normal. If, in addition, $G\times^PN\to G\cdot N$ is a resolution of singularities, then $G\cdot N$ has rational singularities.
\end{corollary}

\begin{proof}
The argument is straightforward from Theorem \ref{maintheorem} with $V=N$. Note that $G\cdot M=M$ a smooth variety, and the completely reducible action of $P$ on $N$ implies the vanishing condition in the theorem. Hence the result follows.  
\end{proof}

\section{Vanishing of induction functor}\label{vanishing}
Before showing applications of the main theorem in the previous section, we prove some vanishing results for induction functor in  certain cases. The calculations in this section are of independent interest as they extend a result in \cite{VX}. The strategy is making use of Koszul resolutions repeatedly for various vector spaces. 

We assume for the rest of the paper that $G=SL_3$, unless otherwise stated. Let $\Phi^+=\{\alpha, \beta, \alpha+\beta\}$ be the set of positive roots of the root system $\Phi$ of $G$. Denote by $X(T)$ the weight lattice of $T$. We further denote by
\[ X(T)^+=\{\lambda\in X(T): (\lambda,\gamma^\vee)\ge 0~\text{for~all~} \gamma\in\Phi^+ \}, \]
the set of dominant weights in $X(T)$.

We first note that Vilonen and Xue have recently showed in \cite{VX} that for all $i, r\ge 1$
\begin{align}\label{VX vanishing}
\opH^i(G\times^B\fraku^r,\calO_{G\times^B\fraku^r})\cong R^i\ind_B^GS(\fraku^{*r})=0.
\end{align}
It follows that the resolution 
\begin{align*}
G\times^B\fraku^r &\to G\cdot\fraku^r,\\
(g,x_1,\cdots,x_r)&\mapsto (g\cdot x_1,\cdots,g\cdot x_r)
\end{align*}
is a ($G$-equivariant) rational resolution.

Now for $\gamma$ either $\alpha$ or $\beta$, we let \[ \mathcal{A}_{\gamma}=X^+\cup \{\mu\in X(T): (\mu,\gamma^\vee)=-1\text{~and~} \mu+\gamma\in X(T)^+ \}.\]
For each simple root, say $\alpha$, we denote $\fraku_{\alpha}$ Lie algebra of the unipotent radical of the parabolic subgroup of $G$ generated by $\{\alpha\}$. It then follows that 
\[ 0\to \fraku_{\alpha}\to\fraku\to -\alpha\to 0. \]
To be convenient, we will write $\fraku^{*r}$ instead of $(\fraku^*)^r$, and write $\opH^i(M)$ instead of $R^i\ind_B^G(M)$ for all $i\ge 0$. We first prove some lemmas.

\begin{lemma}\label{simple lemma}
For $r\ge 0, i\ge 1$ and $\mu\in\mathcal{A}_{\gamma}$, \[ \opH^i\big(S(\fraku_{\gamma}^{*r})\otimes\mu\big)=0.\]
\end{lemma}

\begin{proof}
By symmetry, it suffices to prove the lemma for $\gamma=\alpha$. If $(\lambda,\alpha^\vee)=-1$, then we have the vanishing of $\opH^i\big(S(\fraku_{\alpha}^{*r})\otimes\mu\big)$ for all $i\ge 0$ by a lemma in \cite[3.2]{KLT:1999}. Hence, we assume that $\mu$ is dominant. We prove by induction on $r$. The statement clearly holds for $r=0$. Suppose it holds for $r-1$ for some positive integer $r$. We consider
\[  0\to\fraku_{\alpha}^{*}\to\fraku_{\alpha}^{*r}\to\fraku_{\alpha}^{*(r-1)}\to 0.\]
Tensoring the Koszul resolution of this short exact sequence with $\mu$, we get 
\begin{align}\label{Koszul1}
  0\to S^{n-2}\fraku_{\alpha}^{*r}\otimes\Lambda^2(\fraku_{\alpha})\otimes\mu\to S^{n-1}\fraku_{\alpha}^{*r}\otimes\fraku_{\alpha}\otimes\mu\to S^{n}\fraku_{\alpha}^{*r}\otimes\mu\to S^n\fraku_{\alpha}^{*(r-1)}\otimes\mu\to 0
\end{align}
for all $n\ge 0$. Observe that $\mu+\beta, \mu+(\alpha+\beta),$ and $\mu+(\alpha+2\beta)$ are in $\mathcal{A}_{\alpha}$. Induction on $n$ then implies that for all $i\ge 1$
\[ \opH^i( S^{n-2}\fraku_{\alpha}^{*r}\otimes\Lambda^2(\fraku^*_{\alpha})\otimes\mu)=\opH^i(S^{n-1}\fraku_{\alpha}^{*r}\otimes\fraku^*_{\alpha}\otimes\mu)=0.\]
Now breaking up \eqref{Koszul1} into short exact sequences and applying inductive hypotheses, we obtain $\opH^i(S^{n}\fraku_{\alpha}^{*r}\otimes\mu)=0$ for all $i\ge 1$, which inductively proves our lemma.
\end{proof}
We further extend the above result as follows.

\begin{theorem}\label{alpha}
For all $i\ge 1, r, s\ge 0$, and $\lambda\in\mathcal{A}_{\gamma}$,
\[ \opH^i\big(S(\fraku^{*r}\times\fraku_{\gamma}^{*s})\otimes\lambda\big) =0. \]
\end{theorem}

\begin{proof}
Again, we only have to argue for the case when $\gamma=\alpha$. By Lemma \ref{simple lemma}, the theorem holds for $r=0$. Suppose it holds for $r-1$ and all $s\ge 0$. Proceeding inductively for $n$, we only need to prove that for all $i\ge 1$
\begin{align}\label{r,s alpha}
\opH^i\big(S^n(\fraku^{*r}\times\fraku_{\alpha}^{*s})\otimes\lambda\big) =0.
\end{align}
Assume that $(\lambda,\alpha^\vee)=-1$. Then consider the short exact sequence
\[ 0\to\alpha\to\fraku^{*r}\times\fraku_{\alpha}^{*s}\to\fraku^{*(r-1)}\times\fraku_{\alpha}^{*(s+1)}\to 0.\]
Tensoring the Koszul resolution of this short exact sequence with $\lambda$, we get 
\begin{align*}
  0\to S^{n-1}(\fraku_{\alpha}^{*r}\times\fraku_{\alpha}^{*s})\otimes(\alpha+\lambda)\to S^{n}(\fraku_{\alpha}^{*r}\times\fraku_{\alpha}^{*s})\otimes\lambda\to S^{n}(\fraku_{\alpha}^{*(r-1)}\times\fraku_{\alpha}^{*(s+1)})\otimes\lambda\to 0.
\end{align*}
By inductive hypotheses, the vanishing of $\opH^i(S(\fraku_{\alpha}^{*r}\times\fraku_{\alpha}^{*s})\otimes(\alpha+\lambda))$ implies that of $\opH^i(S(\fraku_{\alpha}^{*r}\times\fraku_{\alpha}^{*s})\otimes \lambda)$ for all $i\ge 1$. Hence, we only have to verify \eqref{r,s alpha} for $\lambda\in X^+$. We consider
\[  0\to\fraku_{\alpha}^{*}\to\fraku^{*r}\times\fraku_{\alpha}^{*s}\to\fraku^{*r}\times\fraku_{\alpha}^{*(s-1)}\to 0.\]
Tensoring the Koszul resolution of this short exact sequence with $\lambda$, we get for all $n\ge 0$
\begin{align}\label{Koszul2}
  0\to S^{n-2}(\fraku^{*r}\times\fraku_{\alpha}^{*s})\otimes\Lambda^2(\fraku^*_{\alpha})\otimes\lambda\to S^{n-1}(\fraku^{*r}\times\fraku_{\alpha}^{*s})\otimes\fraku^*_{\alpha}\otimes\lambda\to \\
\to S^{n}(\fraku^{*r}\times\fraku_{\alpha}^{*s})\otimes\lambda\to S^n(\fraku^{*r}\times\fraku_{\alpha}^{*(s-1)})\otimes\lambda\to 0. \nonumber
\end{align}
As we are assuming $\lambda$ is dominant, $\lambda+\beta, \lambda+(\alpha+\beta),$ and $\lambda+(\alpha+2\beta)$ are in $\mathcal{A}_{\gamma}$. Induction on $n$ then implies that for all $i\ge 1$
\[ \opH^i( S^{n-2}(\fraku^{*r}\times\fraku_{\alpha}^{*s})\otimes\Lambda^2(\fraku^*_{\alpha})\otimes\lambda)=\opH^i(S^{n-1}(\fraku^{*r}\times\fraku_{\alpha}^{*s})\otimes\fraku^*_{\alpha}\otimes\lambda)=0.\]
Now breaking up \eqref{Koszul2} into short exact sequences and applying inductive hypotheses, we obtain $\opH^i(S^{n}(\fraku^{*r}\times\fraku_{\alpha}^{*s})\otimes\lambda)=0$ for all $i\ge 1$, which inductively proves our lemma.
\end{proof}


Setting $s=0$ in the above theorem, we obtain the following

\begin{corollary}\label{lambda in A}
For all $i\ge 1, r\ge 0$ and $\lambda\in\mathcal{A}_{\gamma}$, we have
\[ \opH^i(S(\fraku^{*r})\otimes\lambda)=0. \]
\end{corollary}

\begin{remark}
Our vanishing results in this section holds for all characteristics greater than 3, see \cite[Remark 6.2]{VX}. Note also that analogous vanishings for $S(\frakb^{*r})$ do not hold for $r>1$, see counter-examples in \cite[\S 5.2]{VX}.  
\end{remark}

\section{Nilpotent commuting varieties}\label{normality}

Let $\g$ be a Lie algebra defined over $k$ and $X$ be a closed subvariety of $\g$. For each $r\ge 1$, the commuting variety (of $r$-tuples) over $X$ is defined by
\[ C_r(X)=\{(x_1,\ldots,x_r)\in X^r~|~[x_i,x_j]=0 \} \]
for $r\ge 2$ and $C_1(X)=X$. When $V=\N$, we call $C_r(\N)$ the {\it nilpotent commuting variety}. The study of nilpotent commuting varieties was started not long ago. The pioneering work of Premet \cite{Pr:2003} showed that $C_2(\N)$ has pure dimension\footnote{This means all irreducible components are of the same dimension.} $\dim\g$. In \cite{N:2015}, I could show that the result does not hold for arbitrary $r$. In joint work with \v{S}ivic \cite{NS:2014}, we determined the (ir)reducibility of the variety $C_r(\N)$ with $\g=\fraksl_n$ for various values of $n$ and $r$. Explicitly, it is reducible for all $n, r\ge 4$. Moreover, for $r=3$, it is irreducible for all $n\le 6$.

We keep assuming in this section that $G=SL_3$ defined over $k$. Recall from \cite{N:2012} that various commuting varieties over $2\times 2$ matrices were proved to be Cohen-Macaulay and have rational singularities. While these properties follow easily from determinantal varieties for $C_r(\fraksl_2)$ and $C_r(\frakgl_2)$, the proof for the nilpotent commuting varieties requires deep methods in commutative algebra due to the difficulty in computing their defining ideals. Note also that in \cite[Section 7]{N:2012}, the author showed that the singular locus of $C_r(\N)$ is of codimension 2 which is a strong evidence for the normality of $C_r(\N)$. In this section, we verify the normality of $C_r(\N)$ and prove further that it has rational singularities.

Note that for each $r\ge 1$ we have $C_r(\N)=G\cdot C_r(\fraku)$ and the moment map $m:G\times^B C_r(\fraku)\to C_r(\N)$ is proper birational map, see \cite[Proposition 3.4.3]{N:2012}. We first analyze some properties of $C_r(\fraku)$. First, let $f_{\alpha}, f_{\beta},$ and $f_{\alpha+\beta}$ be root vectors in $\fraku$ corresponding to weights $-\alpha, -\beta,$ and $-\alpha-\beta$. Then each element $v$ in $\fraku$ can be written as 
\[ v=af_{\alpha}+bf_{\beta}+cf_{\alpha+\beta} \]
for some $a, b, c\in k$. Now suppose $(v_i)=\big( a_if_{\alpha}+b_if_{\beta}+c_if_{\alpha+\beta}: 1\le i\le r\big)$, an $r$-tuple in $\fraku^r$. Analyzing the commutator $[v_i,v_j]$ for all $i\ne j$, we get
\[ (v_i)\in C_r(\fraku)\quad \Longleftrightarrow\quad a_ib_j-a_jb_i = 0. \]
These equations are exactly all $2\times 2$ minors of the matrix
\[ \mathcal M=\left( 
\begin{array}{ccc}
a_1 & \cdots & a_r\\
b_1 & \cdots & b_r
\end{array}
\right). \]
It follows that $C_r(\fraku)$ is a product of an affine space and determinantal variety $D_2(\mathcal M)$. Therefore, from Proposition \ref{commutative algebra} we have obtained the following

\begin{proposition}\label{C_r(u)}
For all $r\ge 1$, we have
\begin{enumalph}
\item $C_r(\fraku)$ has rational singularities,
\item the defining ideal of $C_r(\fraku)$ in $S(\fraku^{*r})$ is $I(C_r(\fraku))=\left<f_{i,j}: 1\le i\ne j\le r\right>$ where all $f_{ij}=a_ib_j-a_jb_i\in S(\fraku^{*r})$  are of weight $\alpha+\beta$. Moreover, since $U$ acts trivially on $I(C_r(\fraku))$, we have
\[ I(C_r(\fraku))\cong\bigoplus_{1\le i\ne j\le r}(\alpha+\beta)\otimes S(\fraku^{*r}) \] 
as a $B$-module.
\end{enumalph}
\end{proposition}

Here comes the main result of this subsection.

\begin{theorem}
For all $r\ge 1$, the nilpotent commuting variety $C_r(\N)$ has rational singularities. Consequently, it is Cohen-Macaulay.
\end{theorem}
\begin{proof}
Note first that $C_r(\fraku)$ is a $B$-subvariety of $\fraku^r$. Moreover, $G\cdot\fraku^r$ is shown to have rational singularities in \cite[6.1]{VX}. Hence, our theorem follows immediately from Proposition \ref{equiv resolution}, Theorem \ref{maintheorem}, Corollary \ref{lambda in A}, and Proposition \ref{C_r(u)}.
\end{proof}

\begin{remark}
As pointed out earlier, the result can not be extended further for higher rank groups for all $r$ as it was shown by \v{S}ivic and the author \cite{NS:2014} that $C_r(\N)$ is reducible for all $r\ge 4$ and $\rank(G)\ge 3$ for type $A$ (see also \cite{N:2015} for other classical types).
\end{remark}

The above result can be strengthened as follows.

\begin{theorem}
Let $V$ be a $B$-subvariety of $\fraku^r$ for some $r\ge 1$ whose defining ideal $I(V)$ contains polynomials of weight $\alpha+\beta$. If $V$ has rational singularities and $G\times^BV\to G\cdot V$ is a proper birational map, then $G\cdot V$ has rational singularities.
\end{theorem}

\begin{proof}
Recall from the last section that $G\times^B\fraku^r\to G\cdot\fraku^r$ is a rational resolution. The assertion then follows from Theorem \ref{maintheorem} if we show that $R^i\ind_B^GI(V)=0$ for all $i\ge 1$. This can be done by the same argument as in Proposition \ref{C_r(u)}. Indeed, set $I(V)=\left<f_1,\ldots,f_s \right>$ with all $f_i$'s are of weight $\alpha+\beta$. This implies that $U$ acts trivially on $I(V)$ so that as a $B$-module
\[ I(V)\cong\bigoplus_{i=1}^sf_i\otimes S(\fraku^{*r})\cong\bigoplus_{i=1}^s(\alpha+\beta)\otimes S(\fraku^{*r}). \]
Now Corollary \ref{lambda in A} gives us the desired vanishing of $R^i\ind_B^GI(V)$.
\end{proof}



\section*{Acknowledgments}
The author would like to thank Eric Sommers for discussions about using Koszul resolution and thank Jean-Yves Charbonnel for informing us about work of Vilonen and Xue.

\providecommand{\bysame}{\leavevmode\hbox to3em{\hrulefill}\thinspace}

\end{document}